\documentclass[12pt]{article}
  \usepackage[all]{xy}
 \usepackage{graphicx}
 \usepackage{units}
 
\usepackage{amssymb, amsmath, pict2e}
\title{Tolerance Orders of Open and  Closed Intervals}

\author{Alan Shuchat \\  Department of Mathematics \\ Wellesley College \\
Wellesley, MA 02481 \ \ USA \\ \and
Randy Shull \\ Department of
Computer Science        \\ Wellesley College
\\ Wellesley, MA 02481 \ \ USA \\
\and Ann N.\ Trenk\thanks{This work was supported by a grant from the Simons Foundation (\#426725, Ann Trenk).} \\ Department of Mathematics \\ Wellesley College \\
Wellesley, MA 02481 \ \ USA
}

\date{January 12, 2017}

\newtheorem{theorem}{Theorem}
\newtheorem{defn}[theorem]{Definition}
\newtheorem{example}[theorem]{Example}
\newtheorem{remark}[theorem]{Remark}

\newtheorem{observation}[theorem]{Observation}

\newtheorem{lemma}[theorem]{Lemma}
\newtheorem{prop}[theorem]{Proposition}

\newcommand{\qed}{\mbox{$\Box$}}

\newcommand{\twotwo}{\mbox{${\bf 2}+{\bf 2}$}}

\newcommand{\fourone}{\mbox{${\bf 4}+{\bf 1}$}}
\newcommand{\threeoneone}{\mbox{${\bf 3}+{\bf 1}+{\bf 1}$}}

\newcommand{\threeone}{\mbox{${\bf 3}+{\bf 1}$}}

\begin{document}

  \maketitle

\bibliographystyle{plain}

\begin{center} {\sl ABSTRACT} \end{center}

\begin{quotation}
In this paper we combine ideas from tolerance orders with recent work on OC interval orders.  We consider representations of posets by unit intervals $I_v$ in which the interval endpoints ($L(v)$ and $R(v)$) may be open or closed as well as the center point ($c(v)$).    This yields four types of intervals:  $A$ (endpoints and center points closed), $B$ (endpoints and center points open), $C$ (endpoints closed, center points open), and $D$ (endpoints open, center points closed).  For any non-empty subset $S$ of $\{A,B,C,D\}$, we define an $S$-order as a poset $P$ that has a representation as follows:  each element $v$ of $P$ is assigned a unit interval $I_v$ of type belonging to $S$, and $x \prec y$ if and only if  either (i) $R(x) < c(y)$ or (ii) $R(x) = c(y)$ and at least one of $R(x), c(y)$ is open and at least one of $L(y), c(x)$ is open.    We characterize several of the classes of $S$-orders and provide separating examples between unequal classes.  In addition, for each $S \subseteq \{A,B,C,D\}$ we present a polynomial-time algorithm that recognizes $S$-orders, providing a representation when one exists and otherwise providing a certificate showing it is not an $S$-order.

\end{quotation}


\section{Introduction}
\label{intro-sec}

In this paper we combine ideas from tolerance orders with recent work on OC interval orders.  Both of these concepts start with interval orders and their representations.  A poset  $P=(X,\prec)$ is an \emph{interval order} if each $x \in X$ can be assigned a real interval $I_x$ so that $x \prec y$ if and only if all points of $I_x$  are less than all points of $I_y$.     If such a representation is possible with all intervals having the same length, $P$ is called a \emph{unit interval order}. Throughout this paper we denote the left endpoint of interval $I_x$ by $L(x)$, the right endpoint by $R(x)$ and the center point by $c(x)$.    Tolerance orders are a generalization in which some overlap is allowed between $I_x$ and $I_y$ when $x \prec y$.  More formally, a poset $P=(X,\prec)$ is a \emph{tolerance order} if each $x\in X$ can be assigned a real interval $I_x$ and two \emph{tolerant points\/} $t_\ell(x), t_r(x) \in I_x$ so that $x \prec y$ if and only if all points of $I_x$ are less than $t_\ell(y)$ and all points of $I_y$ are greater than  $t_r(x)$.  We consider the special case in which both tolerant points lie at the center of their interval, that is, $t_\ell(x) = t_r(x) = c(x)$ for all $x \in X$.    These orders are also known as 50\% tolerance orders, first defined for graphs in \cite{BoFiIsLa95}.  For additional background on tolerance orders and their graph analogues, see \cite{GoTr04}.  

Unit OC interval orders are a generalization of  unit interval orders in which each  unit  interval $I_x$ comes in one of two types:  an open interval $(L(x),R(x))$ or a closed interval $[L(x),R(x)]$.   These were first introduced in \cite{RaSz13} in graph form and subsequently studied by other authors, e.g., \cite{DoLePrRaSz12,Jo15,LeRa13,ShShTr16,ShShTrWe14}.   In this paper, we combine the concepts of 50\% tolerance orders and unit  interval orders by labeling the center points in one of two ways, called \emph{open} and \emph{closed}. This leads to four possible types of intervals, illustrated in Table~\ref{ABCD-table}: type $A$ ($I_x$ closed, $c(x)$ closed),  type $B$ ($I_x$ open, $c(x)$ open),  type $C$ ($I_x$ closed, $c(x)$ open),   and type $D$ ($I_x$ open, $c(x)$ closed). The open/closed terminology for points is suggested by the properties in Definition~\ref{ABCD-def}.   Note that unlike endpoints, even when $c(x)$ is called open it is an element of $I_x$.   We consider \emph{unit} OC interval orders because the class without the unit restriction is equivalent to the class of interval orders \cite{RaSz13,ShShTr16}.   Different classes of posets arise from limiting the types of unit  intervals allowed.  

\begin{table}
 \setlength{\unitlength}{0.75pt}  
  \begin{center}
\begin{tabular}{| c  c   c  c  |}
\hline
Type & Interval & Endpoints & Center \\  
\hline \hline 
 & & & \\
$A$& [--------$\bullet$--------]& closed & closed \\  
  \hline
 & & & \\
$B$& (--------$\circ$--------)&  open & open \\ \hline
& & & \\
$C$&[--------$\circ$--------] & closed & open \\ \hline
& & & \\
$D$& (--------$\bullet$--------)& open & closed \\
    \hline
\end{tabular}
 
\end{center}
\caption{The four types of intervals in an $ABCD$-representation.}
\label{ABCD-table}
\end{table}

\begin{defn}
{\rm Let $S$ be a non-empty subset  of $\{A,B,C,D\}$. An $S$\emph{-representation} of a poset $(X,\prec)$ is a collection ${\cal I}$ of unit intervals $I_x, x \in X$, of type belonging to $S$, where $x \prec y$ if and only if

(i) $R(x) < c(y)$ or

(ii) $R(x) = c(y)$,   at least one of $R(x),c(y)$ is open, and at  least one of $L(y),c(x)$ is open. 

An $S$\emph{-order} is a poset with an $S$-representation.  }
\label{ABCD-def}
\end{defn}
 We simplify the notation by eliminating set notation, for example, by referring to a $\{C,D\}$-representation as a $CD$-representation  and   a $\{C,D\}$-order as a $CD$-order.

 It is well-known that interval orders are   those  posets with no induced $\twotwo$ \cite{Fi70}.  However, the poset $\twotwo$ \emph{is} a $CD$-order and a representation is given in 
 Figure~\ref{CD-fig}, where $I_x, I_y$ are type $C$ and $I_z,I_w$ are  type $D$.  We show in Example~\ref{F-example}, that  up to permuting labels, this is the only way to represent $\twotwo$ using unit intervals of type $A,B,C,D$.

 \begin{figure}
 
\begin{center}

\begin{picture}(250,80)(0,0)
\thicklines

\put(5,20){\circle*{6}}
\put(5,65){\circle*{6}}
\put(30,65){\circle*{6}}
\put(30,20){\circle*{6}}
 
\put(5,20){\line(0,1){45}}
\put(30,20){\line(0,1){45}}
 
\put(-8,17){$x$}
\put(-8,62){$y$}
\put(37,17){$z$}
\put(37,62){$w$}

\put(100,80){\line(1,0){100}}
\put(100,55){\line(1,0){100}}
\put(150,30){\line(1,0){100}}
\put(150,5){\line(1,0){100}}

\put(87,52){$z$}
\put(87, 77){$x$}
\put(137,2){$w$}
\put(137,27){$y$}

\put(98,77){[} 
\put(98,52){(} 
\put(148,27){[} 
\put(148,2){(} 

\put(150,80){\circle{6}}
\put(150,55){\circle*{6}}
\put(200,30){\circle{6}}
\put(200,5){\circle*{6}}

\put(198,77){]} 
\put(198,52){)} 
\put(248,27){]} 
\put(248,2){)}

\end{picture}
\end{center}
 
\caption{The poset $\twotwo$  and a $CD$-representation of it.}
\label{CD-fig}
\end{figure}
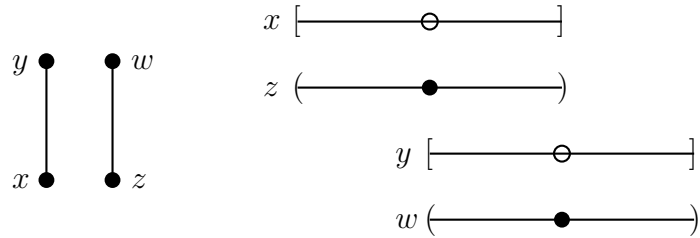

We end this section with two poset definitions that will be important in later sections:  \emph{twin-free} and \emph{inseparable}.  Two points in a poset are said to be \emph{twins} if they have the same comparabilities, and a poset is \emph{twin-free} if it does not contain any twins.  Since twins can be given identical intervals, it suffices to consider twin-free posets when recognizing  classes of $S$-orders.    We say that poset $(X,\prec)$ is \emph{separable} if the ground set $X$ can be partitioned as $X = V \cup W$ so that $v \prec w$ whenever $v \in V$ and $w \in W$; otherwise it is called \emph{inseparable}.   Any poset can be partitioned into inseparable subposets and $S$-representations of these subposets can be joined to give an $S$-representation of the original poset (see \cite{Tr98} for details).  Thus our focus on  inseparable posets in Section~\ref{alg-sec} is not a substantive restriction.  
 
\section{The Case $|S|=1$ and Preliminaries}

 We begin this section with a 
 theorem  that shows that the posets that can be represented using a single type of interval from the set $S$ are precisely the unit interval orders.

\begin{theorem}
For any subset $S$ of $\{A,B,C,D\}$ with $|S|=1$, a poset is an $S$-order if and only if it is a unit interval order.  
\label{singleton-thm}
\end{theorem}

\begin{proof}
One can check that the orders $\twotwo$ and $\threeone$ are not $S$-orders when $|S|=1$, and we provide a short proof of this using forcing cycles in Example~\ref{F-example}.  Using the Scott-Suppes Theorem \cite{ScSu58}, we conclude that all $S$-orders with $|S|=1$ are unit interval orders.

Conversely, let $S$ be a singleton subset of $\{A,B,C,D\}$ and let $P=(X,\prec)$ be a unit interval order. We will prove that $P$ has an $S$-representation. Fix a unit interval representation of $P$ in which all intervals have length $\lambda$ and all endpoints are distinct (see Lemma 1.5 of \cite{GoTr04} for a proof that this is possible). Let $[\ell_x, r_x]$ be the interval assigned to $x \in X$ and let $\epsilon$ be the smallest distance between endpoints in the representation.  For each $x \in X$, define $L(x) = \ell_x - \lambda + \epsilon$, $R(x) = r_x$, and $I_x = [L(x),R(x)]$.   One can check that this gives an $A$-representation of $P$.  Indeed, no center point of this representation is equal to any endpoint, therefore these endpoints also yield $S$-representations of $P$ when 
 $S= \{B\}$,  $S= \{C\}$,  and  $S= \{D\}$.  \qed
\end{proof}

The proof that unit interval orders are $A$-orders also  follows from Theorem~10.4 in \cite{GoTr04}.

\begin{defn} {\rm  Fix an $S$-representation of a poset $P$.  A \emph{CD-swap} occurs when each interval of type $C$ in the representation is transformed into a type $D$ interval with the same center  and each type $D$ interval is similarly transformed into a type $C$ interval.  }
\end{defn}

\begin{lemma}
Let $\cal I$ be an $S$-representation of a poset $P$ and let ${\cal I}'$ be the set of intervals obtained by applying a $CD$-swap to $\cal I$. Then ${\cal I}'$  is also an $S$-representation of   $P$.  
\label{CD-swap}
\end{lemma}

\begin{proof}
${\cal I}'$ is an $S$-representation of a poset $P'$ with the same ground set $X$ as $P$.  We wish to show $P' = P$ by showing that, for all $x,y \in X$, we have $x \prec y$ in $P$ if and only if $x \prec y$ in $P'$.    Consider any two points $x,y$ in $P$ and without loss of generality assume $c(x) \le c(y)$.    If $R(x) < c(y)$ then $x \prec y$ in both $P$ and $P'$, and if $R(x) > c(y)$ then $x \parallel y$ in both $P$ and $P'$.  Hence it suffices to consider the case in which $R(x) = c(y)$.  If either of $I_x, I_y$ is of type $B$, then $x \prec y$  in both $P$ and $P'$.  If one is of type $A$ and the other is not of type $B$,  then $x \parallel y$ in both $P$ and $P'$.   If both $I_x$ and $I_y$ are type $C$ (or both type $D$) then $x \prec y$  in both $P$ and $P'$, and finally, if one of $I_x,I_y$ is of type $C$ and the other is type $D$, then $x \parallel y$ in both $P$ and $P'$.  Thus $P=P'$ and ${\cal I}'$  is   an $S$-representation of   $P$.  $\qed$
\end{proof}

We end this section with an observation that will be useful as we analyze $S$-representations.  Note that the interval  type is only relevant in the third case. In  Observation~\ref{obs-centers}, and throughout the rest of the paper, we will scale representations so that all intervals have length 2.

\begin{observation}  Let $P = (X, \prec)$  be an $S$-order, let ${\cal I} = \{I_x:x \in X\}$ be an $S$-representation of $P$ in which all intervals have length 2, and let   $x,y \in X$.  \begin{itemize}
\item  If $|c(x) -c(y)| < 1$ then $x \parallel y$ in $P$.
\item  If $|c(x) -c(y)| > 1$ and $c(x) < c(y)$, then $x \prec y $ in $P$.
\item  If $|c(x) -c(y)| = 1$  and $c(x) \le c(y)$,
then $x \prec y $ in $P$ precisely when at least one of $I_x,I_y$ is type $B$, or both are type $C$, or both are type $D$.

\label{obs-centers}
\end{itemize}

\end{observation}

\section{Forcing Cycles and Separating Examples} 
\label{forcing-sec}
In this section we use the concept of a forcing cycle,  defined in \cite{GiTr98}, to yield information about $S$-representations.  Forcing cycles are closely related to the \emph{picycles} studied by Fishburn \cite{Fi85}.  We begin with the notion of a \emph{forcing trail} and notation to keep track of the number of comparabilities and incomparabilities encountered along the trail.

\begin{defn} {\rm
A \emph{forcing trail} ${\cal T}$ in poset $(X,\prec)$ is a sequence $x_0,x_1,x_2, \ldots, x_t$ of elements of $X$ so that for each $i: 0 \le i \le t-1$, either $x_i \prec x_{i+1}$ or $x_i \parallel x_{i+1}$.   We also define

$up_{\cal T}(x_i) = | \{j:0 \le j \le i-1, x_j \prec x_{j+1} \} |$

$side_{\cal T}(x_i) = |\{j:0 \le j \le i-1, x_j \parallel x_{j+1} \} |$

$val_{\cal T}(x_i) =  up_{\cal T}(x_i)  - side_{\cal T}(x_i)  $.

A \emph{forcing cycle} is a forcing trail $x_0,x_1,x_2, \ldots, x_t$ for which $x_0 = x_t$, and we define $up({\cal C}) = up_{\cal C}(x_t)$, $side({\cal C}) = side_{\cal C}(x_t)$, and $val({\cal C}) = val_{\cal C}(x_t)$.
}
\label{forcing_trail_def}
\end{defn}

While the first and last elements of a forcing cycle \emph{must} be equal,  there may be other elements in a forcing trail or cycle that are also equal.
For convenience, we sometimes write forcing trails with the comparabilities  and incomparabilities included.  For example,  in the poset $\twotwo$ in Figure~\ref{CD-fig},
we may write the forcing trail ${\cal T}:  x,y,z,w$    as 
${\cal T}:  x \prec y \parallel z \prec w$.    The next lemma shows how forcing trails give lower bounds on centers of elements in $S$-representations.

\begin{lemma}
Let $P$ be an $S$-order and ${\cal T}: x_0, x_1, x_2, \ldots, x_t$ a forcing trail in $P$.  Fix an $S$-representation of $P$ in which all intervals have length 2.   For $0 \le i \le t$,

(i)  $c(x_i) \ge c(x_0) + val_{\cal T}(x_i)$   and 

(ii)  if $c(x_i) > c(x_0) + val_{\cal T}(x_i)$ for some $i$, then $c(x_j) >  c(x_0) + val_{\cal T}(x_j)$ for all $j \ge i$.
\label{val-lem}
\end{lemma}

\begin{proof}
We proceed by induction on $i$.    The base case $i=0$ is true  for (i)  by Definition~\ref{forcing_trail_def},
and true for (ii) vacuously.  We assume (i) and (ii) are true for $i$ and show they hold for $i+1$.  

If $x_i \prec x_{i+1}$ then  by Definition~\ref{forcing_trail_def},
$val_{\cal T}(x_{i+1}) = val_{\cal T}(x_i) + 1 $, and by Definition~\ref{ABCD-def},  $c(x_{i+1}) \ge c(x_i) + 1$.
Hence $c(x_{i+1}) \ge  c(x_0) + val_{\cal T}(x_i) + 1 = c(x_0) + val_{\cal T}(x_{i+1})$.
In addition, if $c(x_i) >  c(x_0) + val_{\cal T}(x_i)$  we obtain $c(x_{i+1}) > c(x_0) + val_{\cal T}(x_{i+1})$.

If $x_i \parallel x_{i+1}$, a similar argument shows that (i), (ii) are true for $i+1$.
\qed
\end{proof}

\smallskip

For some posets $P$, the values of the forcing cycles in $P$ completely determine whether or not $P$ is an $S$-order, and this is independent of the choice of $S$.  For others,  the choice of $S$ is crucial. For example we see from Figure~\ref{CD-fig} that $\twotwo$ is a $CD$-order, yet  we know from Theorem~\ref{singleton-thm}
that $\twotwo$ is not an $S$-order whenever $|S| = 1$.  The next theorem makes this precise.
\begin{theorem}

For a poset $P$, exactly one of the following holds.  

(i)   $P$ has a 
  forcing cycle ${\cal C}:  x_0, x_1, x_2, \ldots, x_t$ with 
 $val({\cal C}) > 0$, in which case  $P$ is not an $S$-order for any $S$.
 
 (ii)  All forcing cycles ${\cal C}$ in $P$ have $val({\cal C}) < 0$, in which  case $P$ is a unit interval order and an $S$-order for all  non-empty $S$.
 
 (iii)  No forcing cycle in  $P$  has a positive value and there exists a   forcing cycle with value 0.   For any value 0 forcing cycle ${\cal C}:  x_0, x_1, x_2, \ldots, x_t$,  any $S$-representation of $P$ in which all intervals have length 2 must have $c(x_i) = c(x_0) + val_{\cal C}(x_i)$, for all $i: 0 \le i \le t$.
\label{wd-thm}
\end{theorem}

\begin{proof}
To prove (i),   let  ${\cal C}:  x_0, x_1, x_2, \ldots, x_t$ be a forcing cycle in $P$ with $val({\cal C}) >0$.  Thus $val_{\cal C}(x_t) > 0$.  
Suppose for a contradiction that $P$ is an $S$-order for some $S$.
    Fix an $S$-representation of $P$ in which all interval lengths are 2.  Then  $c(x_t) > c(x_0)$ by part (i) of Lemma~\ref{val-lem}, but this is a contradiction since $x_t = x_0$.

    Next we prove (ii) relying on prior results.  By hypothesis we know $up({\cal C}) < side({\cal C})$ for each forcing cycle ${\cal C}$ in $P$, thus $\max_{\cal C} \frac{up({\cal C})}{side({\cal C})} < 1$,  where the maximum is taken over all forcing cycles in $P$.  By Theorem~13 in \cite{ShShTr07}, the  fractional weak discrepancy of $P$ is less than 1, and by Proposition~10 in \cite{ShShTr06}, $P$ is a semiorder, another term for a unit interval order.    It then follows from Theorem~\ref{singleton-thm}    that $P$ is an $S$-order for all $S$ with $|S|$ = 1 and thus for all non-empty $S$.

    Finally, we prove (iii).  Let  ${\cal C}:  x_0, x_1, x_2, \ldots, x_t$ be a forcing cycle in $P$ with $val({\cal C}) =0$.   Suppose for a contradiction that $P$ has an $S$-representation in which  all intervals have length 2, but for which there exists $i$ so that $c(x_i) >  c(x_0) + val_{\cal C}(x_i)$.  By (ii) of Lemma~\ref{val-lem},  we have $c(x_t) > c(x_0) + val_{\cal C}(x_t)$, a contradiction since $c(x_t) =c(x_0)$.
     \qed
\end{proof}

\smallskip
For any $S$, Theorem~\ref{wd-thm} completely determines whether  a poset $P$ satisfying (i) or (ii) is an $S$-order.  
Thus, for the rest of this paper we focus on the remaining posets, that is, those that satisfy (iii).  There are posets, such as $V$ of Example~\ref{F-example}, that satisfy (iii) yet are not $S$-orders for any $S$. 
Theorem~\ref{sewing} shows that the key to determining if   a poset $P$ satisfying (iii)  is an $S$-order is being able to find an $S$-representation for each value 0  forcing cycle  in $P$.   If $S$-representations for these forcing cycles exist, they can be interlaced together.    We use Theorem~\ref{sewing} in the proof of Proposition~\ref{venn-prop} and defer the proof of the theorem until Section~\ref{alg-sec},
where we present an algorithm to achieve the interlacing.

\begin{theorem}
\label{sewing}
Let $S$ be a non-empty subset of $\{A,B,C,D\}$.  Suppose $P$ is a poset for which $val({\cal C}) \le 0$ for every forcing cycle ${\cal C}$ in $P$.  Furthermore, suppose the points of every forcing cycle with value 0 induce in $P$ an $S$-order.  Then $P$ is an $S$-order.
\end{theorem}
 We conclude this section by considering several of the posets shown in Figures~\ref{separating-posets} and \ref{F-posets}, using forcing cycles to characterize those $S$ for which they are $S$-orders.

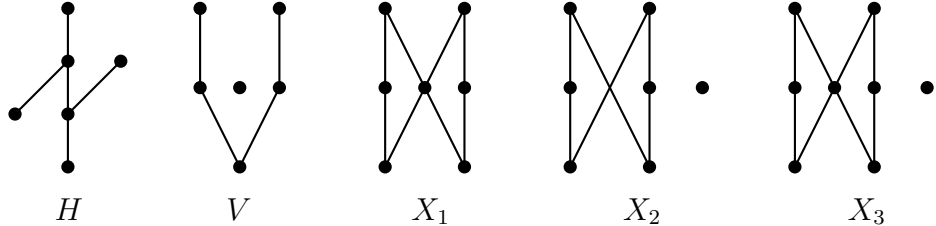
\begin{figure}
\begin{center}
 \begin{picture}(350,100)(0,0)
\thicklines

 \put(20,20){\circle*{5}}
\put(20,40){\circle*{5}}
\put(20,60){\circle*{5}}
\put(20,80){\circle*{5}}
\put(0,40){\circle*{5}}
\put(40,60){\circle*{5}}

\put(20,20){\line(0,1){60}}
\put(20,40){\line(1,1){20}}
\put(20,60){\line(-1,-1){20}}
 \put(15,0){$H$}

  \put(85,20){\circle*{5}}
\put(70,50){\circle*{5}}
\put(85,50){\circle*{5}}
\put(100,80){\circle*{5}}
\put(70,80){\circle*{5}}
\put(100,50){\circle*{5}}

\put(85,20){\line(1,2){15}}
\put(85,20){\line(-1,2){15}}
\put(70,80){\line(0,-1){30}}
\put(100,80){\line(0,-1){30}}
 \put(80,0){$V$}

 \put(140,20){\circle*{5}}
\put(140,50){\circle*{5}}
\put(140,80){\circle*{5}}
 \put(170,20){\circle*{5}}
\put(170,50){\circle*{5}}
\put(170,80){\circle*{5}}
\put(155,50){\circle*{5}}
\put(140,20){\line(0,1){60}}
\put(170,20){\line(0,1){60}}
\put(140,20){\line(1,2){30}}
\put(170,20){\line(-1,2){30}}
 \put(150,0){$X_1$}

 \put(210,20){\circle*{5}}
\put(210,50){\circle*{5}}
\put(210,80){\circle*{5}}
 \put(240,20){\circle*{5}}
\put(240,50){\circle*{5}}
\put(240,80){\circle*{5}}
\put(260,50){\circle*{5}}
\put(210,20){\line(0,1){60}}
\put(240,20){\line(0,1){60}}
\put(210,20){\line(1,2){30}}
\put(240,20){\line(-1,2){30}}
 \put(230,0){$X_2$}
 
 \put(295,20){\circle*{5}}
\put(295,50){\circle*{5}}
\put(295,80){\circle*{5}}
 \put(325,20){\circle*{5}}
\put(325,50){\circle*{5}}
\put(325,80){\circle*{5}}
\put(345,50){\circle*{5}}
\put(310,50){\circle*{5}}
\put(295,20){\line(0,1){60}}
\put(325,20){\line(0,1){60}}
\put(295,20){\line(1,2){30}}
\put(325,20){\line(-1,2){30}}
\put(315,0){$X_3$}

  \end{picture}

\end{center}
\caption{Five separating examples for the Venn diagrams in Figures~\ref{Venn-2-fig} and \ref{Venn-3-fig}.}

\label{separating-posets}
 \end{figure}

\begin{figure}
\begin{center}
 \begin{picture}(300,100)(0,0)
\thicklines

\put(20,20){\circle*{5}}
\put(20,40){\circle*{5}}
\put(20,60){\circle*{5}}
\put(20,80){\circle*{5}}
\put(40,50){\circle*{5}}
\put(20,20){\line(0,1){60}}
 \put(5,0){\fourone}

\put(80,30){\circle*{5}}
\put(80,50){\circle*{5}}
\put(80,70){\circle*{5}}
\put(95,50){\circle*{5}}
\put(110,50){\circle*{5}}
\put(80,30){\line(0,1){40}}
 \put(60,0){\threeoneone}
 
 \put(150,20){\circle*{5}}
\put(150,40){\circle*{5}}
\put(150,60){\circle*{5}}
\put(150,80){\circle*{5}}
\put(130,40){\circle*{5}}
\put(170,60){\circle*{5}}

\put(150,20){\line(0,1){60}}
\put(150,20){\line(1,2){20}}
\put(150,80){\line(-1,-2){20}}
 \put(145,0){$Z$}

  \put(220,20){\circle*{5}}
\put(220,80){\circle*{5}}
\put(205,50){\circle*{5}}
\put(235,50){\circle*{5}}
\put(250,50){\circle*{5}}

\put(220,20){\line(1,2){15}}
\put(220,20){\line(-1,2){15}}
\put(205,50){\line(1,2){15}}
\put(235,50){\line(-1,2){15}}
 \put(220,0){$D$}

  \put(300,20){\circle*{5}}
\put(300,50){\circle*{5}}
\put(285,80){\circle*{5}}
\put(315,80){\circle*{5}}
\put(315,50){\circle*{5}}

\put(300,20){\line(0,1){30}}
\put(300,50){\line(1,2){15}}
\put(300,50){\line(-1,2){15}}
 
 \put(295,0){$Y$}

  \end{picture}

\end{center}
\caption{Forbidden posets which, with the dual of $Y$, comprise the set ${\cal F}$ of Theorem~\ref{AB-charac}.}

\label{F-posets}
 \end{figure}
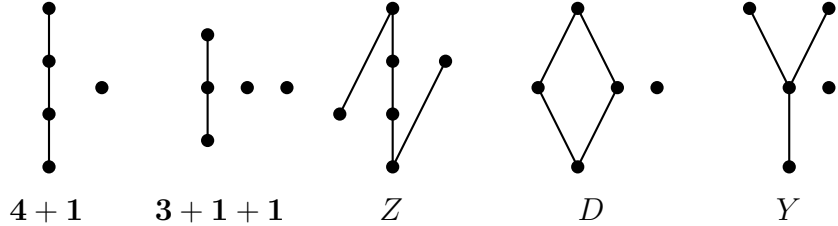

 \begin{example}  {\rm   In this example, we consider several posets and provide detailed proofs that show each is correctly positioned in the Venn diagrams of Figures~\ref{Venn-2-fig} and \ref{Venn-3-fig}.  Section~\ref{appendix} contains the proofs for the remaining posets appearing in these 
Venn diagrams.
 
\medskip

 \noindent  {\bf The poset $\twotwo$.}  This poset has the forcing cycle ${\cal C}: x \prec y \parallel z \prec w \parallel x$ with $val({\cal C}) = 0$.  Suppose $\twotwo$ were an $S$-order and without loss of generality, fix an $S$-representation $\cal I$ in which each interval has length 2 and $c(x) =0$.           Theorem~\ref{wd-thm} implies $c(y) = 1, c(z) = 0$ and $c(w) = 1$.      Using Observation~\ref{obs-centers}, none of $I_x, I_y,I_z,I_w$ can be type $B$ because   $x \parallel w$ and $z \parallel y$, and therefore,  because $x \prec y$ and $z \prec w$, none of these intervals can be type $A$.    The only possibility  using intervals of types $C$ and $D$ is for $I_x$ and $I_y$ to be type $C$ and $I_z$ and $I_w$ to be type $D$ as in Figure~\ref{CD-fig} (or vice versa).

 \medskip
 
\noindent {\bf The poset $\threeone$. } This poset has the forcing cycle ${\cal C}: x \prec y  \prec z \parallel u  \parallel x$ with $val({\cal C}) = 0$.  Suppose $\threeone$ were an $S$-order and without loss of generality, fix an $S$-representation $\cal I$ in which each interval has length 2 and $c(x) =0$.       Now Theorem~\ref{wd-thm} implies $c(y) = 1, c(z) = 2$ and $c(u) = 1$.    Since $x \parallel u$ and $z \parallel u$, by Observation~\ref{obs-centers}, none of the intervals $I_x, I_z, I_u$ can be type $B$.  Observation~\ref{obs-centers} now implies that $\threeone$ is not a $BC$-order, since that would require $I_y$ to be type $B$ and $I_x, I_z, I_u$ to be type $C$. Similarly, $\threeone$ is not a $BD$-order. 
   It is not hard to check that $\threeone$ \emph{is} an $S$-order for all other $S$ with $|S| = 2$.  

 \medskip
 
\noindent {\bf   The poset $\fourone$. } This poset has the forcing cycle ${\cal C}:  x \prec y \prec z \prec w \parallel v \parallel x$ with $val({\cal C}) = 1$.  By Theorem~\ref{wd-thm},    $\fourone$ is not an $S$-order for any $S$.
 
 \medskip

\noindent {\bf   The poset $V$. }  All forcing cycles in $V$ have value at most 0, yet $V$ is not an $S$-order for any $S$. Since $V$ has an induced $\twotwo$, if it were an $S$-order,  without loss of generality, the four elements of $\twotwo$ would have the representation given in Figure~\ref{CD-fig}.  It is easy to check that this representation cannot be extended to an $ABCD$-representation of all of $V$.

    \medskip
    
 \noindent  {\bf  The poset $Z$.}  This poset has the forcing cycle ${\cal C}:  x \prec y \prec z \prec w \parallel v \parallel u \parallel x$ with $val({\cal C}) = 0$.     Suppose $Z$ were an $S$-order  for some $S$.  Without loss of generality, fix an $S$-representation  $\cal I$ of $Z$ in which all intervals have length 2 and     $c(x) = 0$.  Now Theorem~\ref{wd-thm} implies that $c(y) = 1, c(z) = 2, c(w) = 3, c(v) =2, c(u) = 1$.  

Each element of $Z$ is incomparable to another element where the two centers  differ by 1, so by  Observation~\ref{obs-centers}, none of the intervals in any $S$-representation can be of type $B$.   Thus it suffices to consider representations using intervals of types $A$, $C$, and $D$. 
  A representation is possible if $S= \{A,C\}$ (namely by making $I_x,I_y,I_z, I_w$ of type $C$,  and $I_u, I_v$ of type $A$).   By  Lemma~\ref{CD-swap} it follows that  a representation is also possible for $S= \{A,D\}$. 
 However, $Z$ is not a $CD$-order as we now show.
 If  there were a $CD$-representation of $Z$, then without loss of generality we may assume $I_x$ is type $C$.  It then follows that $I_y$, $I_z$ and $I_w$ are also of type $C$ and $I_u$ and $I_w$ must be type $D$, a contradiction since $u \parallel v$.

 }
 \label{F-example}
 \end{example}

\begin{figure}
\begin{center}
\includegraphics[height=3in]{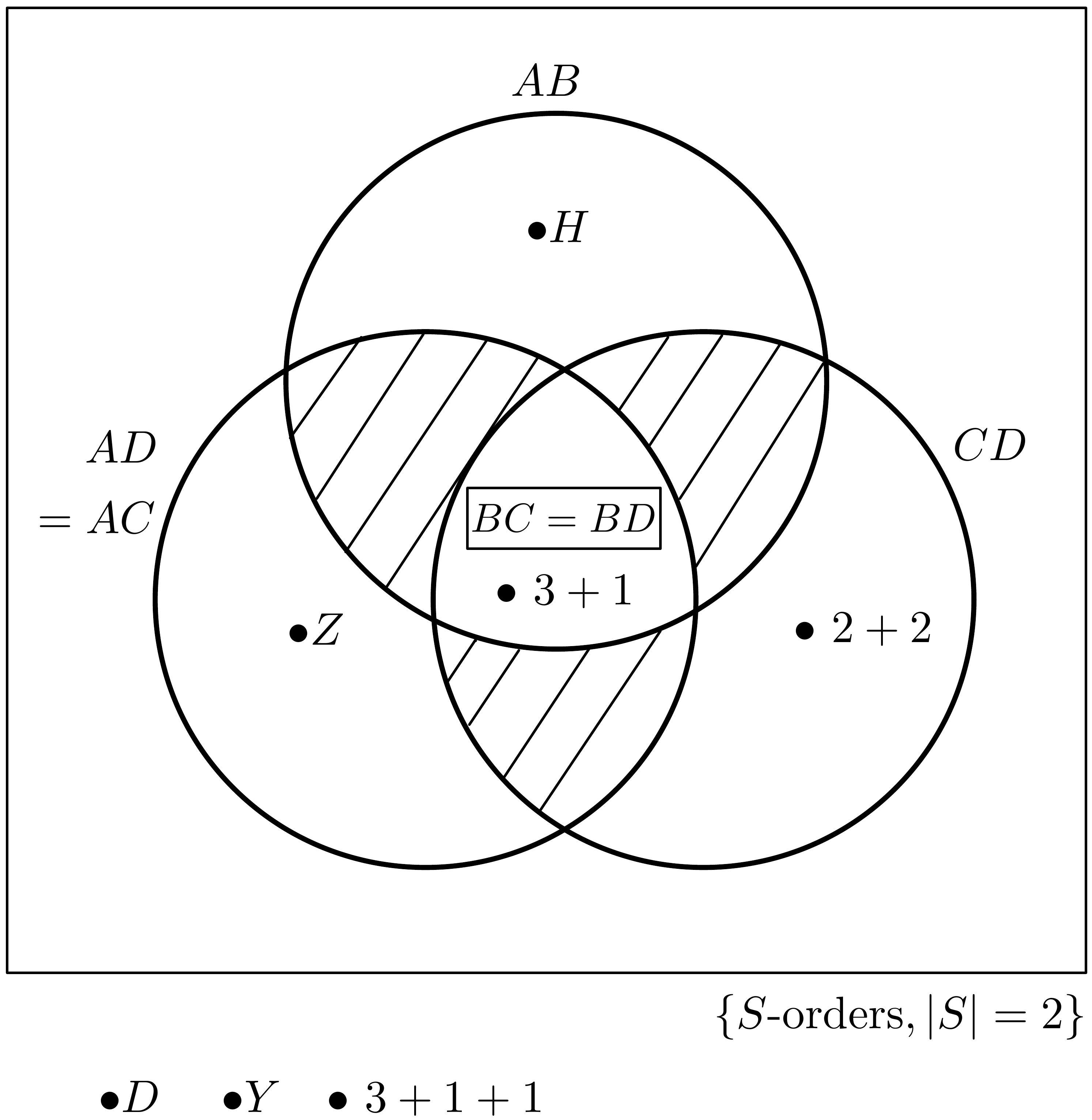}
\end{center}
\caption{Venn diagram showing  classes of twin-free $S$-orders when $|S| = 2.$}
\label{Venn-2-fig}

\end{figure}

\begin{figure}
\begin{center}
\includegraphics[height=3in]{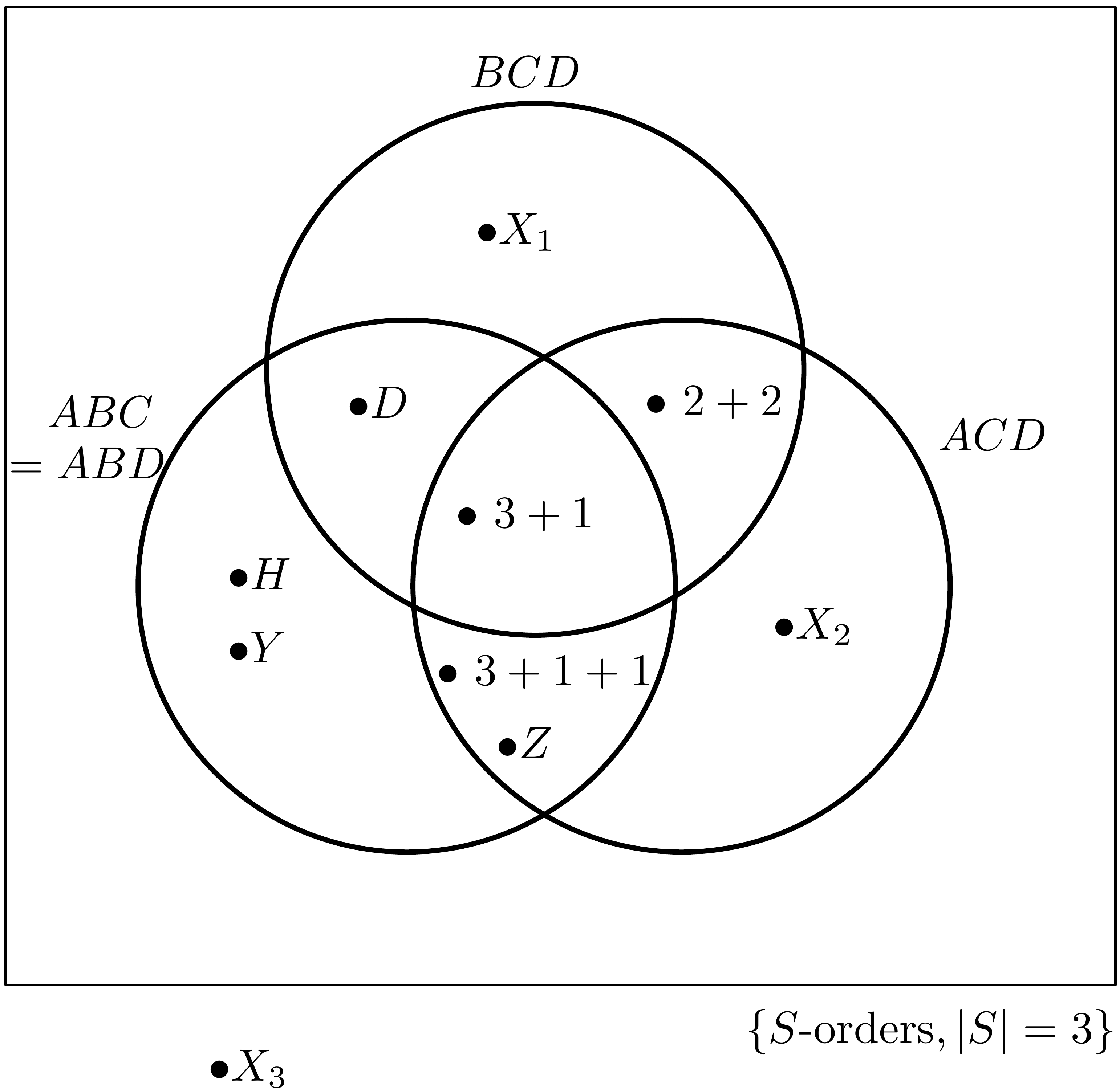}
\end{center}
\caption{Venn diagram showing  classes of twin-free $S$-orders when $|S| = 3.$}
\label{Venn-3-fig}
\end{figure}

\section{The case $|S| = 2$.}
\label{s-equals 2-sec}

Because of Lemma~\ref{CD-swap}, there are four classes of $S$ orders with $|S| = 2$, namely $AB$, $AC=AD$, $BC = BD$, and $CD$.  These are shown in Figure~\ref{Venn-2-fig}.   Proposition~\ref{BC-charac-prop} characterizes $BC$-orders and Theorem~\ref{AB-charac}
characterizes $AB$-orders.

\begin{prop} A poset is a $BC$-order if and only if it is a unit interval order.
\label{BC-charac-prop}
\end{prop}

\begin{proof}
By Theorem~\ref{singleton-thm}, we know that unit interval orders are $BC$-orders.  Conversely, if $P$ is a $BC$-order, it contains neither an induced $\twotwo$ nor an induced $\threeone$, as shown in Example~\ref{F-example}.  Thus, by the Scott-Suppes Theorem \cite{ScSu58}, $P$ is a unit interval order.
 $\qed$

\end{proof}

\begin{theorem}  The following are equivalent for a twin-free poset $P$.
\begin{enumerate}
\item  $P$ is a unit OC interval order.
\item  $P$ is an $AB$-order.
\item  $P$  is an interval order that does not contain as an induced poset any element of the set $\cal F$ consisting of the five  posets   of Figure~\ref{F-posets} and the dual of $Y$. 
\end{enumerate} 
 \label{AB-charac}
\end{theorem}
\begin{proof}
Since unit OC interval orders are interval orders, the equivalence of (1) and (3) is shown in Theorem~12 of  \cite{ShShTr16}. 
 
 \smallskip

 \noindent
 $(1) \Longrightarrow (2)$.    Let $ P= (X,\prec)$ be a unit OC interval order and fix a unit OC interval representation  of $P$  in which each interval has length $\lambda$. Let $\ell_x, r_x$ be the endpoints of the interval assigned to $x \in X$ and let $\epsilon$ be the smallest distance between distinct endpoints in the OC representation. We transform these intervals as follows.  Let $L(x) = \ell_x - (\lambda - \epsilon/4)$ and   $R(x) = r_x + \epsilon/4$. For each $x \in X$, let the interval $I_x$ have endpoints $L(x), R(x)$, and note that 
 $I_x$ has length  $2\lambda$. Define $I_x$  to be of type $A$ ($B$) if the corresponding OC interval was closed (open).  One can then check that the intervals $I_x$ provide an $AB$-representation for $P$.
  
 \smallskip
 
 \noindent
 $(2) \Longrightarrow (3)$.  
 Let $P$ be a twin-free $AB$-order.  Since $\twotwo$ is not an $AB$-order, poset $P$ is an interval order.    In Example~\ref{F-example}  and Section~\ref{appendix} we show $P$ cannot contain an induced poset in $\cal F$, i.e., any of the five  posets   of Figure~\ref{F-posets} or the dual of $Y$.  \qed
\end{proof}

\medskip

The next result is a technical lemma that describes the structure of value 0 forcing cycles that can exist in certain twin-free    $S$-orders   where $|S| \le 2$.    The statement of Lemma~\ref{xyz-lemma}
refers to poset $H$ shown in Figure~\ref{separating-posets}
and poset $Z$ shown in Figure~\ref{F-posets}.
 We will often simplify our notation here and in the rest of the paper when the meaning is clear, e.g., ``$v$ is type $C$''  means that $I_v$ is a type $C$ interval.

\begin{lemma}
Let $P$ be a twin-free $S$-order that satisfies one of the following:  (i) a $CD$-order with no induced $\twotwo$,   (ii)  an $AB$-order with no induced $H$, or (iii) an $AC$-order with no induced $Z$.  Then a forcing cycle $\cal C$ has value 0 in $P$ precisely when it has the following form.  The set of elements in $\cal C$ can be written as the union of three sets  $\{x_i: i = 0,2,4, \ldots, 2q\} \cup  \{y_i: i = 1,3,5, \ldots, 2q-1\} \cup \{z_i: i = 1,3,5, \ldots, 2q-1\}$ for some $q$.  Furthermore, for consecutive elements $u,v$ of $\cal C$, 
\begin{align*}
&\hbox{$u \prec v$ if and only if for some $i$, ($u = x_i$, $v = y_{i+1}$) or ($u = y_i$, $v = x_{i+1}$);}
\\
&\hbox{$u \parallel v$ if and only if for some $i$, ($u = x_i$, $v = z_{i+1}$) or ($u = z_i$, $v = x_{i+1}$). }
\end{align*}
\label{xyz-lemma}
\end{lemma}

\begin{proof}  Let ${\cal C}: v_0, v_1, v_2,  \ldots, v_n$ be a forcing cycle in $P$ with $val({\cal C}) = 0$.   Fix an $S$-representation of $P$ and without loss of generality, assume each interval has length 2 and the minimum value of $c(v_j)$ is $c(v_0)$ and equals 0.  By  Theorem~\ref{wd-thm},   for all $j$ we have  
$c(x_j) =   val_{\cal C}(x_j)   $, so $c(v_{j+1}) = c(v_j) + 1$ if $v_j \prec v_{j+1}$ and $c(v_{j+1}) = c(v_j) - 1$ if $v_j \parallel v_{j+1}$. Thus the values of $c(v_j)$ form the set $\{0,1,2,\ldots, t\}$ for some integer $t$.   As a result, we have the following:
\begin{align}
& \hbox{for  $0 \le i \le t-1$, there exists $j$ with $c(v_j) = i$ and $v_j \prec v_{j+1}$;}
\label{eq-1}
  \\
 & \hbox{for  $1 \le i \le t$, there exists $k$ with $c(v_k) = i$ and  $v_k \parallel v_{k+1}$.}
\label{eq-2}
 \end{align}
   
 Combining these, we see that in the $S$-representation restricted to the elements of $\cal C$, there cannot be consecutive integers in $\{0,1,2,\ldots, t\}$  where each is the center of  only one interval.
  
Now consider  case  (i) where $S = \{C,D\}$ and $P$ has no induced $\twotwo$.  Using  Lemma~\ref{CD-swap}, we may assume $v_0$ is type $ C$.   Since $v_n = v_0$ and $c(v_0) = 0$ is the minimum value for $c(v_j)$, we have  $v_0 \prec v_1$, $v_{n-1} \parallel v_n$.  Thus $v_1$ is  type $C$ and $v_{n-1}$ is type $D$, both with center 1.  If there were two consecutive centers \emph{each} with a type $C$ and a type $D$ interval, then these four elements would induce a $\twotwo$ in $P$, a contradiction.  Therefore  in any $S$-representation of the poset induced by the elements of $\cal C$, there must be two intervals with center $i$ when $i$ is odd, and one interval with center $i$ when $i$ is even.  Because of this structure and the choice to make $v_0$ type $C$,  whenever $v_j \prec v_{j+1}$ on $\cal C$ both $v_j$ and $v_{j+1}$ must be type $C$.    Thus   for $i = 0,1,2, \ldots, t$, there is a type $C$ interval at center $i$, and for $i = 1,3,5, \ldots, t-1$, there is a type $D$ interval, and hence $t$ is even.    We now relabel the elements of $\cal C$ according to their interval types and centers as follows.  
When $i = 0,2,4, \ldots t$, let $x_i$ be the type $C$ interval, and when $i = 1,3,5, \ldots t-1$, let $y_i$  (resp. $z_i$) be the type $C$ (resp. $D$) interval.  This gives the desired form. 

Next consider  case (ii) where $S = \{A,B\}$ and $P$ has no induced poset isomorphic to $H$.  By (\ref{eq-2}), there exists a type $A$ interval with center $i$ for $i = 0,1,2,\ldots, t$.  Let $x_i$ be those with $i$ even and $z_i$ be those with $i$ odd.   For $1 \le j \le n-1$, if $c(v_j) = 0$  (resp. $c(v_j) = t$) then $v_{j-1} \parallel v_j$ (resp. $v_j \parallel v_{j+1}$) and $v_j$ cannot be type $B$.  Since $P$ is twin-free, there is thus only one element of $\cal C$ with center 0 and one  with center $t$ and both of these are type $A$.  If there were consecutive centers $i, i+1$ each with a type $A$ and a type $B$ interval, then those four elements together with  the type $A$ intervals at centers $i-1$ and $i+2$  would induce the poset $H$ in $P$, a contradiction.  Thus type $B$ intervals must exist at center $i$ for $i = 1,3,5, \ldots t-1$ and $t$ must be even.  Let $y_i$ be the type $B$ interval with $c(y_i) = i$.  This gives the desired form.

Finally, consider  case (iii) where $S = \{A,C\}$ and $P$ has no induced poset isomorphic to $Z$.  By (\ref{eq-1}), there exists a type $C$ interval with center $i$ for $i = 0,1,2,\ldots, t$.  Let $x_i$ be those with $i$ even and $y_i$ be those with $i$ odd.   For $1 \le j \le n-1$, if $c(v_j) = 0$  (resp. $c(v_j) = t$) then $v_{j} \prec v_{j+1}$ (resp. $v_{j-1} \prec v_{j}$) and $v_j$ cannot be type $A$.  Thus there is only one element of $\cal C$ with center 0 and one  with center $t$ and both of these are type $C$.  If there were consecutive centers $i, i+1$ each with a type $A$ and type $C$ interval, then those four elements together with the type $C$ intervals at centers $ {i-1}$ and $ {i+2}$ would induce the poset $Z$ in $P$, a contradiction.  As in case (ii), type $A$ intervals exist only at center $i$ for $i = 1,3,5, \ldots t-1$ and $t$ is even.  Let $z_i$ be the type $A$ interval with $c(z_i) = i$.  This gives the desired form.
\qed
\end{proof}

The following corollary shows that the shaded regions of Figure~\ref{Venn-2-fig} are empty.  Furthermore, it shows that each of  the posets $\twotwo$,   $Z$, and $H$ that appear in Figure~\ref{Venn-2-fig}  is the unique minimally forbidden poset in its region of that Venn diagram. 

\begin{prop}
  (i)  Any $CD$-order  with no induced $\twotwo$ is  an $AB$-order and an $AC$-order.
    
  (ii)  Any $AC$-order with no induced poset isomorphic to $Z$ is an $AB$-order and  a $CD$-order.

(iii)  Any $AB$-order with no induced poset isomorphic to $H$ is a $CD$-order and an $AC$-order.

Moreover, the shaded regions of Figure~\ref{Venn-2-fig} are empty.
\label{venn-prop}
\end{prop}

\begin{proof}  
To prove (i), suppose $P$ is an $CD$-order with no induced $\twotwo$.  By (i) of Theorem~\ref{wd-thm},
 $val({\cal C}) \le 0$ for every forcing cycle  $\cal C$ in $P$.    By Theorem~\ref{sewing}, it suffices to prove that for any forcing cycle $\cal C$ in $P$ with $val({\cal C}) = 0$, the poset induced by the points of $\cal C$ is both an $AB$-order and an $AC$-order. Using 
 Lemma~\ref{xyz-lemma},  we know the form of forcing cycle $\cal C$.  If we let $x_i$  and $z_i$ be    type $A$, and $y_i$ be  type $B$, we get an $AB$-representation of the  poset induced by the elements of $\cal C$.  Likewise, if we let $x_i$ and $y_i$ be type $C$ and $z_i$ be type $A$, we get an $AC$-representation of this poset.  
  The proofs of (ii) and (iii) are similar.  From (i), (ii), and (iii) we conclude that the shaded regions of Figure~\ref{Venn-2-fig} are empty. \qed
 
\end{proof}

\section{Recognizing $S$-orders}
\label{alg-sec}

In this section we present an algorithm that recognizes $S$-orders, providing a representation when one exists and a certificate otherwise.
A related algorithm appears in \cite{Tr98}.  As noted at the end of Section~\ref{intro-sec},
the restrictions to posets that are twin-free and  inseparable is not substantive.

 We begin with an overview of the algorithm.  In forming an $S$-representation of a poset $P$, each element is assigned a center and a type.  The latter is accomplished in a procedure called \emph{Assign Types}, which we present after the main algorithm.
    To initialize the algorithm, a starting element $v_0$ is chosen and its center $c(v_0)$ is fixed.  The center $c(v_i)$ of each remaining element $v_i$ is assigned a lower bound $\ell(v_i)$ and an upper bound $u(v_i)$.   
    The algorithm proceeds in passes where lower bounds may increase and upper bounds may decrease,  but at all times, if an $S$-representation is possible, then one exists in which each center lies within its bounds.

   Each pass starts with an element whose center is fixed.  In a given pass, we either identify a forcing cycle with value greater than zero and terminate the algorithm, or we identify all points that lie on a value 0 forcing cycle that also includes the starting point of that pass.  
   
   If such a value 0 forcing cycle ${\cal C}$ exists, the centers of its elements are fixed and the procedure \emph{Assign Types} determines if the elements on ${\cal C}$ have an $S$-representation with these centers.  If not, the algorithm terminates.  Once a center is fixed during a pass, it remains fixed during all subsequent passes.
     
  For elements $x, y$ whose centers are fixed in different passes of the algorithm, achieving the appropriate relationship between them ($x \prec y, \ y \prec x$, or $x \parallel y$) is independent of the types of intervals used and depends only on the placement of $c(x)$ and $c(y)$.   In pass $r$, $V_r$ is the set of elements $v_i$ with $\ell(v_i) < u(v_i)$ at the start of pass $r$.  During pass $r$, changes are made in these bounds by making pairwise comparisons with the bounds of other elements of $V_r$, during the \emph{Labeling Loop}. We use a $\{0,1\}$-matrix $M$ and tracking functions $f,g$ for this purpose. In particular, $f(v_j) = v_i$ means that $\ell(v_j)$ was most recently changed by considering the ordered  pair $(v_i,v_j)$, and $g(v_j) = v_i$ similarly means that $u(v_j)$ was most recently changed.
When $M_{ij} = 0$, this signals that    the pair $(v_i, v_j)$ must   be considered  (again) in the narrowing process.

In pass $r$, if the labeling loop terminates without the entire algorithm terminating, we let  $X_r$ be the set of elements $v_j$ in $V_r$ for which $\ell(v_j) = u(v_j)$.  For these elements, we set $c(v_j) = \ell(v_j)$ and run procedure \emph{Assign Types}. If the algorithm continues to pass $r+1$, the elements in $X_r$ no longer need to be actively considered and we let $V_{r+1} = V_r - X_r$.

  \medskip
  
  \noindent
  {\bf Algorithm  Recognize $S$-Orders}
  \medskip

\noindent
{\bf Input:}  A twin-free, inseparable  poset   $P = (V,\prec)$  and a non-empty subset $S$ of $\{A,B,C,D\}$.

\noindent
{\bf Output: }  Either an $S$-representation of $P$ or a certificate showing that $P$ is not an $S$-order.

\noindent
{\bf Data Structure needed:}
An $|V|  \times |V|$ array $M$  whose entries are either 0 or 1.    In pass $r$, only the first $|V_r|$ rows and columns of $M$ are used.
     
       \smallskip

\noindent
{\bf Procedure Narrowing Steps (\emph {NS}):} (to be used repeatedly in the \emph{Labeling Loop}):

\begin{enumerate}
\item If $v_i \prec v_j$
 and $\ell(v_j) < \ell(v_i) + 1$,
increase $\ell(v_j)$  to $\ell(v_i) + 1$ and  set $f(v_j) = v_i$. 
 \item If $v_i \parallel v_j$
   and $\ell(v_j) < \ell(v_i) - 1$, 
  increase $\ell(v_j)$   to $\ell(v_i) - 1$ and  set $f(v_j)  = v_i$.
\item If $v_i \succ v_j$
  and $u(v_j) > u(v_i) - 1$, 
 decrease $u(v_j)$ to $u(v_i) - 1$ and set $g(v_j) = v_i$.
\item  If $v_i \parallel v_j$
  and $u(v_j) > u(v_i) + 1$,
  decrease $u(v_j)$  to $u(v_i) + 1$ and set $g(v_j) = v_i$.
\end{enumerate}

\noindent
{\bf Start of Algorithm:}

Set $r = 0$. Let $V_0 = V,$ and choose some $v_0 \in V_0$.  Set $ c(v_0) :=0$, $ \ell(v_0) :=0$,  and $ u(v_0) :=0$.

For each $v \in V_0, v \ne v_0$, set $\ell(v) = -|V| -1$ and $u(v) = |V| + 1$.

\medskip
\noindent
{\bf Start of Pass $r, r \geq 0$:}

Let $n_r = |V_r|$ and label the remaining elements of $V_r$ so that  $V_r = \{v_i: i = 0, 1, 2, 3, \ldots, n_r\}$. 

For $j = 1,2,3, \ldots, n_r$, set $f(v_j) = g(v_j) = nil$. 

For $i,j$ satisfying $1 \le i,j \le n_r$, set \[M_{ij} := \left\{ \begin{array}{ll}
  0   & \mbox{ for all  $i \neq j$} \\
  1 & \mbox{ if $i=j$   .} 
  \end{array}
  \right. \]  
  
  Run Procedure \emph{NS} for $i=0$ and $j = 1,2,3, \ldots, n_r$.

  \medskip

\noindent
{\bf Labeling Loop:}

Choose $(i,j)$ with $1 \le i,j \le n_r$ and  $M_{ij} = 0$.  If no such pair exists,  let $X_r := \{v_i \in V_r: \ell(v_i) = u(v_i)\}$ and  exit the {\it Labeling Loop}.

Run Procedure \emph{NS} for the pair $(i,j)$.

If $\ell(v_j) > u(v_j)$, end the entire algorithm and report that $P$ is not an $S$-order for any $S$.  (By Proposition~\ref{trace-thm}, $P$ is not an $S$-order, and the proof of that theorem shows how to produce a forcing cycle with value greater than 0.)

If neither $\ell(v_j)$ nor $u(v_j)$ is changed, set $M_{ij} = 1$.
Otherwise, set the nondiagonal entries in row $j$ and column $j$ of $M$ to 0.

Repeat the {\it Labeling Loop}.

  \medskip
  
\noindent
{\bf Procedure Assign Types (details following Example~\ref{proc-example}):}

Set $c(x) = \ell(x)$ for each $x \in X_r$.   

Run Procedure \emph{Assign Types} (with elements ordered by their indices) to produce either an $S$-representation of    $P$ restricted to the elements of $X_r$ or a certificate that no such representation exists.   In the latter case, terminate the algorithm.  

\medskip

\noindent
{\bf Preparing for Next Pass:}

Let $V_{r+1} = V_r - X_r$. Designate an element of $V_{r+1}$ to be $v_0$. Choose  $c(v_0) = \ell(v_0) + \frac{1}{2^{r+1}}$ and  update $\ell(v_0) := c(v_0)$ and $u(v_0) := c(v_0)$.

Begin pass $r+1$.

\medskip

 \noindent
 {\bf End of Algorithm:}
 
The algorithm terminates when all elements of $P$ have a designated center and type, producing an $S$-representation of $P$, or terminates during either the \emph {Labeling Loop} or Procedure \emph{Assign Types}, proving that no such representation exists.

Before presenting the procedure for assigning types, we define compatible type assignments and provide an example.  

\begin{defn} {\rm
Let $S$ be a subset of $\{A,B,C,D\}$.  Let $Q$ be a poset together with an assignment of a real number center to each element of $Q$ and, for each 
 $j$, a
 list of the  $t_j$ elements of $Q$ with center $j$.
A \emph{type assignment for center $j$} is a
 list of $t_j$ distinct elements of $S$; this determines an interval for each element of $Q$ with center $j$ by matching corresponding entries in the two lists.     Type assignments for centers $j$ and $j+1$ are \emph{compatible} if the assigned intervals provide an $S$-representation for the poset induced in $Q$ by the $t_j + t_{j+1}$ elements   with centers at $j$ or $j+1$.

}
\label{defn-assign}
\end{defn}

We illustrate this definition with the following example.  

\begin{example}
{\rm Consider the poset $Z$ given in Example~\ref{F-example} where $t_0 = 1, t_1 = t_2 = 2$ and $t_3= 1$.  Order the elements of $Z$ according to  the forcing cycle ${\cal C}: x \prec y \prec z \prec w \parallel v \parallel u \parallel x$.  This induces an order of the elements at each center (e.g., $yu$ at center 1, $zv$ at center 2).  The type assignment $CA$ at center 1 is compatible with $CD$ at center 2, but not with $DC$ at center 2.  The type assignment $CD$ at center 2 is compatible with $C$ at center 3 but not with any of $A, B,$ or $D$.
}

\label{proc-example}
\end{example}
\medskip
  
  \noindent
  {\bf Procedure Assign Types}
  \medskip
  
  \noindent
  {\bf Input: } A subset $S$ of $\{A,B,C,D\}$, a value 0  forcing cycle ${\cal C}: x_0,x_1, x_2, \ldots, x_n$,  the poset $Q$ induced by the elements of $\cal C$, and a center $c(x_i)$ assigned to each $x_i$ so that $c(x_i) = c(x_0) + val_{\cal C}(x_i)$.
  
  \noindent
  {\bf Output:}  Either an $S$-representation of $Q$ or the conclusion that $Q$ is not an $S$-order.
  
  The elements with center $j$ are ordered according to their first occurrence in $\cal C$.   
  Let $m = \min\{c(x_i): 0 \le i \le n\}$ and $M =  \max \{c(x_i): 0 \le i \le n\}$. (By construction, $M-m$ is an integer.)
  For each $j = m, m+1, m+2, \ldots, M$, let $t_j$ be the number of elements of $Q$ with center $j$.  If $t_j > |S|$ for any $j$, terminate the procedure and report that $Q$ is not an $S$-order.  Otherwise, for each $j$, create a list of nodes, one for each ordered list of $t_j$ distinct elements of $S$.  Each node consists of three fields:   (i) a type assignment, (ii) a set of pointers (initially empty) back to nodes at $j-1$, and  (iii) a set of pointers (initially empty) forward  to nodes at $j+1$. 
  
  \smallskip
  
  \noindent
  {\bf Initialize} $j := m$.
  
    \smallskip
    
  \noindent
  {\bf Loop:}
  For each  node $T_1$ at center $j$ and each node $T_2$ at center $j+1$,  if the type assignments are compatible, add a forward pointer from $T_1$ to $T_2$ and a backward pointer from $T_2$ to $T_1$.
  Delete all  nodes at center $j+1$ with no backward pointers. 
  
   If there are no nodes remaining at center $j+1$, terminate the procedure, report that $Q$ is not an $S$-order, and return the forcing cycle $\cal C$.
  
   If $j+1 < M$, increment $j$ and begin the loop again.
   
   If $j+1 = M$, an $S$-representation of $Q$ exists.  One can be obtained by starting at a node with center $M$ and following backward pointers through nodes at each center to obtain an $S$-representation.  \qed

  \medskip
 
 \begin{example} 
 {\rm 
  We continue Example~\ref{proc-example} where $m = 0$ and $M = 3$. For simplicity, we refer to each node at a given center by its type assignment.  The node $C$ at center 0 has forward links to nodes $BA, BD, CA$, and $CD$ at center 1, but neither  $BA$ nor $BD$ has forward links to nodes at center 2.  At center 3 there are 2 nodes with  backward links ($C$ and $D$), leading to a total of eight possible paths back to a node at center 0.  One such path gives the type assignments $C$ (center 0), $CA$ (center 1), $CD$ (center 2), $C$ (center 3), and the resulting $S$-representation in which $I_x, I_y, I_z, I_w$ are type $C$, $I_v$ type $D$, and $I_u$ type $A$.
  }
  \end{example}
  
  \smallskip
  
  We now present results to justify the correctness and complexity of Algorithm \emph{Recognize $S$-Orders}.
The next lemma ensures that after pass 0,  if the algorithm has not terminated, each lower and upper bound is finite.

 \begin{lemma} If Algorithm Recognize $S$-Orders does not terminate in pass 0, then at the end of pass 0, $\ell(v)$ and $u(v)$ are integers and   
 $-|V| \le \ell(v)\le u(v) \le |V|$ for all $v \in V$.
 \label{finite-lem}
 \end{lemma}
 
 \begin{proof}  Let $Y$ be the set of all elements $y$ of $P$ for which there is a forcing trail in $P$ from $v_0$ to $y$.  During pass 0, using steps $NS$ 1 and 2, each element $y \in Y$ has $\ell(y)$ increased from $- |V| - 1$ to at least $-|V|$.   By the definition of $Y$, for all $x \in V-Y$ and all $y \in Y$, we have $x \prec y$.  Since $P$ is inseparable, we must have $V-Y = \emptyset$, thus $\ell(v) \ge -|V|$ at the end of pass 0 for all $v \in V$.  A symmetric argument shows $u(v) \le |V|$, and we know $\ell(v) \le u(v)$ since the algorithm did not terminate during pass 0.  By Procedure $NS$, we know $\ell(v)$ and $u(v)$ are integers.
 \qed
 \end{proof}

\begin{prop}
If $P$ contains a forcing cycle with value greater than 0, then Algorithm Recognize $S$-Orders will terminate during pass 0 and return a forcing cycle with value greater than 0.
\label{detect-fc}
\end{prop}

\begin{proof}
Let ${\cal C}: x_0, x_1, x_2,  \ldots x_t$ be a forcing cycle in poset $P$ that has value greater than 0, and suppose for a contradiction that Algorithm \emph{Recognize $S$-Orders} continues to pass 1.   Therefore, Procedure $NS$ will be applied to  pairs $(v_0,v_j)$ for $j = 1,2,3, \ldots, n_0$ and to the pairs $(v_i,v_j)$ for $1 \le i,j \le n_0$ when $i \neq j$ until $M_{ij} = 1$ for all $i \neq j$.  

First consider the case in which $v_0$ is an element of $\cal C$ and without loss of generality  assume $v_0 = x_0$.  We will show that under these assumptions, we eventually have $\ell(x_{t-1}) > u(x_{t-1})$, a contradiction to our assumption that the algorithm does not terminate during pass 0.
We wish to show  that at the end of pass 0,
\begin{equation}
 \ell(x_i) \ge \ell(x_0) + val_{\cal C}(x_i)   \hbox{ for all } i  \hbox{ with }  0 \le i \le t.  
 \label{val-eqn}
\end{equation}

 For a contradiction, let $j$ be the smallest integer for which $\ell(x_j) < \ell(x_0) + val_{\cal C}(x_j)$  and note that $j \ge 1$.  Thus $\ell(x_{j-1}) \ge \ell(x_0) + val_{\cal C}(x_{j-1})$.    If $x_{j-1} \prec x_j$ we have
$$\ell(x_j) < \ell(x_0) +  val_{\cal C}(x_{j}) = \ell(x_0) +  val_{\cal C}(x_{j-1}) + 1 \le \ell(x_{j-1}) + 1.$$
When the lower bound of $x_{j-1}$ received its final value $\ell(x_{j-1})$ in pass 0, the matrix entry $M_{j-1,j}$ was set to 0 and Procedure $NS$ was applied to the pair $(x_{j-1}, x_j)$.  Since $\ell(x_j) < \ell(x_{j-1}) + 1$,  in this iteration of Procedure $NS$  the value of $\ell(x_j)$ would have increased, a contradiction.  We get a similar contradiction in the case $x_{j-1} \parallel x_j$. 

Applying (\ref{val-eqn}) when $j = t-1$ yields $\ell(x_{t-1}) \ge 0 +  val_{\cal C}(x_{t-1})$.
If $x_{t-1} \prec x_t$ then $val_{\cal C}(x_{t-1}) > -1$ and  $\ell(x_{t-1}) \ge val_{\cal C}(x_{t-1}) > -1$.  Applying Procedure $NS$ to $(x_0,x_{t-1})$ yields $u(x_{t-1}) \le -1$, thus   $\ell(x_{t-1}) > u(x_{t-1})$ and the algorithm would terminate in pass 0, a contradiction.  A similar contradiction is reached in the case  $x_{t-1} \parallel x_t$.

 Now consider the case that $v_0$ is not an element of $\cal C$. 
Let $\ell$ be the value of $\ell(x_0)$ when pass 0 ends. 
  As in the previous case,  $\ell(x_i) \ge \ell  +  val_{\cal C}(x_{i}) $ for $i $ satisfying $0 \le i \le t$.  When $i=t$ this yields $\ell(x_0) = \ell(x_t) \ge \ell  +  val_{\cal C}(x_{t}) > \ell $, a contradiction.     \qed

\end{proof}

 The next proposition is a technical result showing that when Algorithm \emph{Recognize  $S$-Orders}  terminates during the \emph{Labeling Loop}, there is a forcing cycle with value greater than 0.  Remark~\ref{nil-remark} and Definition~\ref{delta-def} help to simplify the proof.

\begin{remark}{\rm
In pass $r$ of Algorithm \emph{Recognize $S$-Orders},  after each application of Procedure $NS$, the following are equivalent for  $i \in \{1,2,3, \ldots, n_r\}$:

(a)  $f(v_i) \neq nil$

(b) $\ell(v_i)$ has been changed in pass $r$.

(c)  $f(v_i) = v_0$ or $f^2(v_i) \neq nil$.

An analogous statement is true for $g(v_i), u(v_i)$.
}
\label{nil-remark}
\end{remark}

\begin{defn}
Let $\cal {T}:$ $x_0, x_1, x_2, \ldots, x_t$ be a forcing trail.    For $0 \le k \le t-1$, define  $\delta_k = 1$ if $x_{k } \prec x_{k+1}$ and $\delta_k = -1$ if $x_{k} \parallel x_{k+1}$.   In addition, define $\delta_t = 1$ if $x_t \prec x_0$ and  $\delta_t = -1$ if $x_t \parallel x_0$. 
\label{delta-def}
\end{defn}

Note that by definition, $val({\cal T}) = \sum_{k=0}^{t-1} \delta_k$.  In the case that $x_t \prec x_0$ or $x_t \parallel x_0$, the forcing cycle  $\cal C$: $x_0,x_1, x_2, \ldots, x_t, x_0$ has $ val({\cal C}) =\sum_{k=0}^{t}\delta_k$.

\begin{prop}  If Algorithm Recognize $S$-Orders is run on an inseparable poset $P$ and terminates during the Labeling Loop, then there exists a forcing cycle $\cal C$ in $P$ with $val({\cal C}) > 0$.  Consequently, $P$ is not an $S$-order for any $S$.
\label{trace-thm}
\end{prop}

\begin{proof}   
Suppose the algorithm  terminates during pass $r$ of the \emph{Labeling Loop} when  $\ell(v_m) > u(v_m)$ for some element $v_m$ of $P$.      We know that  at least one of $\ell(v_m), u(v_m)$ has changed during pass $r$, so by symmetry we may assume that $\ell(v_m)$ has changed and thus increased.      By Remark~\ref{nil-remark}, it is well-defined to  apply the function $f$ iteratively starting at $v_m$,  until this sequence either terminates at $v_0$ or repeats.   Let  $\hat{\cal T}$  be the resulting  sequence  $v_m, f(v_m), f^2(v_m), \ldots$.  By the relations in steps $NS$, the reverse of any segment of $\hat{\cal T}$ is a forcing trail in $P$. 

 First suppose $\hat{\cal T}$ has a repeating element, so that its reverse contains a forcing cycle $\cal C$. Let ${\cal C}:  x_0, x_1, x_2, \ldots, x_t, x_0$, where   $x_0, x_1, x_2,  \ldots, x_t$ are distinct elements.   
 By construction, we know $f(x_k) = x_{k-1}$ for each $k \ge 1$ and $f(x_0) = x_{t}$.  
 Without loss of generality,  we may assume $x_t$ is the last of these elements to have its lower bound increased.  Let $\ell$ be the value of the lower bound of $x_t$ just before this final increase, thus $\ell < \ell(x_t)$.   When the ordered pair $(x_t,x_0)$ is last considered by Procedure $NS$, the lower bound of $x_0$ is increased, because $f(x_0) = x_t$.   At that time, the lower bound of $x_t$  has value at most $\ell$ because $x_t$ is the last  element on $\cal C$ to have its lower bound increased. Thus,
 the lower bound of $x_0$ satisfies  $\ell(x_0) \le \ell + \delta_t$.

   By Procedure $NS$, we have  $\ell(x_{k+1}) \le \ell(x_k)    + \delta_k \  $ for $0 \le k \le t-1$.  (Indeed, equality holds when the narrowing step is applied to the pair $(x_k,x_{k+1})$ but it is possible that $\ell(x_k)$ increased subsequently.)    Summing these inequalities for $k = 0, 1,2,  \ldots, t-1$ and subtracting  $\sum_{k=1}^{t-1} \ell(x_k)$ from both sides yields the first inequality below.  Substituting $\ell(x_0) \le \ell + \delta_t$  yields the second.
   
   $$ \ell(x_t) \le \ell(x_0) + \sum_{k=0}^{t-1} \delta_k \le \ell + \sum_{k=0}^t \delta_k.$$
   Now   replacing $\sum_{k=0}^{t}\delta_k $ by $val({\cal C})$  yields
   $\ell(x_t) - \ell \le val({\cal C})$.  Since $\ell(x_t) > \ell$, we have   produced a forcing cycle (${\cal C}$) with $val({\cal C}) > 0$.

Next, suppose the sequence $\hat{\cal T}$ has no repeating elements.  Using Remark~\ref{nil-remark} and the fact that   $f(v_i)$ is   defined for all $i$ with $0 < i \le n_r$, we conclude that  the sequence $\hat{\cal T}$ must end at $v_0$, so  $f^t(v_m) = v_0$ for some $t$.  Let $\cal T$ be the reverse of $\hat{\cal T}$ and write $\cal T$ as the sequence $v_0=x_0, x_1, x_2, \ldots, x_t=v_m$.  Now $\cal T$ is a forcing trail from $v_0$ to $v_m$ and by construction,  $\ell(x_0) = \ell(v_0) = c(v_0)$.  We consider the values of the lower and upper bounds when the algorithm terminates with $\ell(v_m) > u(v_m)$.    As before, by Procedure $NS$, we have  $\ell(x_{k+1}) \le \ell(x_k) +    \delta_k $ for $0 \le k \le t-1$.  Summing these inequalities for $k = 0,1,2, \ldots, t-1$, replacing $\sum_{k=0}^{t-1} \delta_k$  by $val ({\cal T})$, and subtracting  $\sum_{k=1}^{t-1} \ell(x_k)$ from both sides, we obtain

\begin{equation}
\label{ell-ineq}
val({\cal T}) \ge \ell(x_t) - \ell(x_0) = \ell(v_m) - c(v_0).
\end{equation}


We next consider the upper bounds and first show that $u(v_m)$  \emph{must} have decreased during pass $r$.  Suppose for a contradiction that $u(v_m)$ has not changed during pass $r$.  Since the initial value assigned to $u(v_m)$ is $|V| + 1$, by Lemma~\ref {finite-lem}
we know  $r \ge 1$.   The element labeled $v_0$ in pass $r$ had its lower and upper bounds set to $c(v_0)$ at the end of pass $r-1$.  Let $\ell, u$ be the values of its lower and upper bounds  in pass $r-1$ just before the step \emph{Preparing for Next Pass}.  Thus at 
  the beginning of pass $r$ we have  $\ell < c(v_0) < u$.  We continue to use the forcing trail ${\cal T}: x_0, x_1, x_2, \ldots, x_t$ where $v_0 = x_0$ and $v_m = x_t$.    At the end of pass $r-1$, because of  Procedure $NS$, we know that  $u(x_k) \le u(x_{k+1}) - \delta_k$ for $k \in \{0,1,2, \ldots, t-1\}$.    Summing these for $k = 0,1,2, \ldots, t-1$ and subtracting  $\sum_{k=1}^{t-1} u(v_k)$ from both sides,  we obtain
$$u(x_0) \le u(x_t) - \sum_{k=0}^{t-1} \delta_k = u(x_t) - val(\cal T).$$

Recall that $x_0 = v_0$ and the value of $u(v_0)$  just before the step \emph{Preparing for Next Pass} at the end of pass $r-1$ is $u$.  Thus $u \le u(x_t) - val(\cal T)$.  By assumption, the value of $u(x_t)$ is unchanged in pass $r$ and $\ell(x_t) > u(x_t)$, so $u < \ell(x_t) - val(\cal T)$.    However, (\ref{ell-ineq}) implies that $\ell(x_t) \le c(v_0) + val(\cal T)$, thus $u < c(v_0)$.  This contradicts the choice of $c(v_0)$  as satisfying $\ell < c(v_0) < u$.  Therefore, we conclude that $u(v_m)$ \emph{must} have decreased during pass $r$.

Now the rest of the argument for upper bounds is similar.  If the sequence $v_m, g(v_m), g^2(v_m), \ldots, $ has a repeated element, as above we get a forcing cycle  with value greater than 0.  Otherwise, the sequence is a forcing trail $\cal R$ from $v_m$ to $v_0$ with $val({\cal R}) \ge c(v_0) -u(v_m)$.  Concatenating $\cal T$ and $\cal R$ yields a forcing cycle $\cal C$ with $val({\cal C}) = val({\cal T}) + val({\cal R}) \ge \ell(v_m) -u(v_m) > 0$.

    Theorem~\ref{wd-thm} now shows that $P$ is not an $S$-order for any $S$.  \qed 
\end{proof}

The next proposition ensures that the indexed set $X_r$ specified at the end of the \emph{Labeling Loop} in pass $r$ provides a valid input to Procedure \emph{Assign Types.}

\begin{prop}
If Algorithm Recognize $S$-Orders does not terminate during the labeling loop of pass $r$ then
at the conclusion of pass $r$,   the set of points $X_r$ lie on a forcing cycle with value 0.
 \label{Xr-forcing-cycle}
\end{prop}

\begin{proof}
Pick any $v_j$ in $X_r$.   By definition of $X_r$, we know $\ell(v_j) = u(v_j)$ at the end of pass $r$.   We will show $v_j$ lies on a forcing cycle with $v_0$ that has value 0.  

As in the proof of Proposition~\ref{trace-thm},  consider the sequences $v_j, f(v_j), f^2(v_j), \ldots$ and $v_j, g(v_j), g^2(v_j), \ldots$.  If either has a repeated element, the proof of Proposition~\ref{trace-thm} shows there exists  a forcing cycle $\cal C$ with $val({\cal C}) > 0$.   By Proposition~\ref{detect-fc},  Algorithm \emph{Recognize $S$-Orders} would terminate during  the labeling loop in pass 0.

Thus neither sequence has a repeated element.  Following the proof of Proposition~\ref{trace-thm}, let $\cal T$ be a forcing trail from $v_0$ to $v_j$ with $val({\cal T}) = \ell(v_j) - c(v_0)$ and $\cal R$ be a forcing trail from $v_j$ to $v_0$ with $val({\cal R}) = c(v_0) - u(v_j)$.  Concatenate $\cal T$ and $\cal R$
  to obtain a forcing cycle $\cal C$  that includes $v_0$ and $v_j$.  Now $val({\cal C}) = val({\cal T}) + val({\cal R}) =   \ell(v_j) - u(v_j) = 0.$   Thus each $v_j \in X_r$ lies on a forcing cycle with $v_0$ that has value 0 and the   resulting forcing cycles can be concatenated to obtain one forcing cycle of value 0 that contains all points in $X_r$.
   \qed
\end{proof}

Interval types are only important in an $S$-representation when the centers of points differ by exactly 1.  We next define what it means for a pair $x,y$ with fixed centers to be \emph{type independent}  and later show that if the centers for two elements are assigned in different passes of the algorithm, then that pair of  elements is type independent.   As a result, when assigning interval types, we need only consider elements whose centers are assigned in the same pass.

\begin{defn} {\rm Let $(X,\prec)$ be a poset,    $x,y$ be elements of $X$, and $c(x), c(y)$ be  real numbers   assigned to $x, y$ respectively.     We say that the pair $x,y$  is \emph{type independent}   if the following hold:\\
(i) If $x \prec y$ then $c(x) + 1 < c(y)$\\
(ii) If $y \prec x$ then $c(y) + 1 < c(x)$ \\
(iii)  If $x \parallel y$ then $|c(x) -c(y)| < 1$.    }
\end{defn}

\begin{lemma}  Suppose Algorithm Recognize $S$-Orders is run on an $S$-order $(X,\prec)$.  Let   $y \in X_r$ and $x \in X_i$ for some $i < r$, and let $c(x)$ be the center assigned to $x$ at the end of pass $i$.  If $L = \ell(y)$ and $U = u(y)$ at the start of pass $r$ and $c(y)$, the value assigned at the end of pass $r$, satisfies     $L < c(y) < U$   then the pair  $x,y$  is type independent.
\label{compatible-lem}
\end{lemma}
 
\begin{proof}  
By the definition of $X_r$ and $X_i$, at the start of pass $r$  we have $c(x) = \ell(x) = u(x)$ and $\ell(y) < u(y)$.  Furthermore, at the end of pass $i$ no additional narrowing steps are employed thus (i) if $x \prec y$ then $\ell(y) \ge \ell(x) + 1$,  (ii) if $y \prec x$ then $u(y) \le u(x) - 1$,   and (iii) if $x \parallel y$ then $u(y) - u(x)  \le 1$ and $\ell(y) - \ell(x) \ge -1$.    These inequalities remain true at the end of pass $r-1$ since $\ell(x), u(x)$ do not change and $\ell(y)$ can only increase and $u(y)$ can only decrease.  Hence, if $x \prec y$ then $c(y) > \ell(y) \ge \ell(x) + 1 = c(x) + 1$.  If $y \prec x$ then $c(y) < u(y) \le u(x) -1 = c(x) -1$.  Finally, if $x \parallel y$ then $c(y) - c(x) < u(y) - u(x) \le 1$ and $c(y) - c(x) > \ell(y) -\ell(x) \ge -1$, thus $|c(x) -c(y)| < 1$.  We conclude that the pair  $x,y$  is type independent.
 \qed
\end{proof}

\begin{theorem}
Algorithm  Recognize $S$-Orders  correctly determines whether poset $P$ is an $S$-order.  In the affirmative it produces an $S$-representation of $P$.  In the negative, it produces a certificate:  either a forcing cycle with value greater than 0 or a forcing cycle with value 0 for which Procedure Assign Types fails.
\label{alg-correct}
\end{theorem}

\begin{proof}
If the algorithm terminates during the \emph{Labeling Loop}, the proof of Proposition~\ref{trace-thm}
 shows how to recover a forcing cycle with value greater than 0.  Thus by Theorem~\ref{wd-thm}, $P$ is not an $S$-order for any $S$.  If the algorithm terminates during Procedure \emph{Assign Types}, then  that procedure returns a value 0 forcing cycle for which there is no $S$-representation.
  
 Otherwise, the algorithm terminates with a representation in which the center of $x$ is $c(x)$ for all $x$ in $P$ and we will show this is an $S$-representation of $P$.   Recall that the ground set of $P$ is partitioned into $\{X_r\}$ where $X_r$ consists of the points $x$  for which $c(x)$ is defined during pass $r$.  First we consider two points in the same part of this partition.
 By Proposition~\ref{Xr-forcing-cycle}, the points of $X_r$ are part of a forcing cycle ${\cal C}_r$ with $val({\cal C}_r) = 0$,  and by Theorem~\ref {wd-thm}, fixing $c(x_0)$ for some $x_0 \in X_r$ determines $c(x_i)$ for all $x_i \in X_r$.    The procedure \emph{Assign Types} determines whether $P$ restricted to $X_r$ is an $S$-order and this is independent of the value chosen for $c(x_0)$.    Thus the poset restricted to $X_r$ is an $S$-order for each $r$ and the algorithm produces $S$-representations for each part.  
 
 Finally, we show that  any pair of points in different parts of the partition is type independent, thus regardless of the type of intervals assigned in procedure \emph{Assign Types}, the representation is an $S$-representation of $P$.
 Consider two points in different parts of the partition,  $x$ in $X_i$ for some $i < r$ and   $y \in X_r$.    At the beginning of pass $r$, the set $V_r$ consists of the elements $w$  of $P$ for which $c(w)$ has not yet been defined.  These are precisely the elements for which $\ell(w) < u(w)$.    By construction, $\ell(w)$ and $u(w)$ are integer multiples of $\frac{1}{2^{r-1}}$ and thus  $u(w) - \ell(w) \ge \frac{1}{2^{r-1}}$.   In pass $r$, an element $v_0 \in V_r$ is chosen and its center is set to $c(v_0) = \ell(v_0) + \frac{1}{2^r}$.  Thus $\ell(v_0) < c(v_0) < u(v_0)$.  As the narrowing steps are implemented in pass $r$,  any lower or upper bound \emph{that is changed}, is changed to an integer of the form $\frac{k}{2^{r-1}} + \frac{1}{2^r}$ for some $k$.  Thus at the end of pass $r$ we will have $\ell(y) = u(y) = \frac{k}{2^{r-1}} + \frac{1}{2^r}$ for some integer $k$.  Thus  when $c(y)$ is assigned a value, it is strictly between   the lower and upper bounds at the beginning of the pass.  By Lemma~\ref{compatible-lem}, the pair $x,y$ is type independent.  \qed 
\end{proof}

\begin{theorem}
Algorithm Recognize $S$-Orders runs in  $O(n^5)$ time on a poset with $n$ elements.
\end{theorem}

\begin{proof}
First we consider Procedure \emph{Assign Types}.  The minimum and maximum values of centers ($m$ and $M$) are determined using a linear number of comparisons.  After this, the procedure creates an array of lists, one list for each distinct center.  Each list contains at most 24 nodes and each node contains  a type assignment, a list of at most 24 backwards pointers   and a list of at most 24  forward pointers.  Each of the pairs of nodes at centers $j$ and $j+1$ is checked for type compatibility in $O(1)$ time.  Thus, all comparisons between nodes at adjacent centers require $O(1)$.  All elements of the poset participate in exactly one invocation of the procedure.  Thus, the total running time of all invocations of  Procedure \emph{Assign Types} is bounded by $O(n)$.

As shown in Lemma~\ref{finite-lem}, 
  the initial finite values assigned to $\ell(v_i)$ and $u(v_i)$  in pass 0 must be between $-n$ and $n$.  After this, each time a lower bound changes during pass 0, it increases by at least one and each time an upper bound changes, it decreases by at least one.

In pass $r$ for $r \ge 1$, a point $v_0$ is selected with $\ell(v_0) < u(v_0)$ and assigned center $c(v_0) = \ell(v_0) + \frac{1}{2^r}$ and its new lower and upper bounds are each given value $c(v_0)$.    Any changes that occur during pass $r$ will result in lower and upper bounds being assigned a value that differs from $c(v_0)$ by an integer.  Thus if $\ell(v_i)$ is increased more than once, it increases by at least one after the first change, and similarly, upper bounds are decreased by at least one after the first change.    There are $n$ elements overall, and each element has a bound changed a total of $O(n)$ times, thus there are $O(n^2)$ changes in bounds before the algorithm terminates.  There are also $O(n^2)$ comparisons between changes in bounds, hence the algorithm requires $O(n^4)$ comparisons and arithmetic operations.   Each bound is represented using $O(n)$ bits, thus the overall running time is $O(n^5)$.
\qed
\end{proof}

We end this section by providing  a proof of Theorem~\ref{sewing}, which was stated in Section~\ref{forcing-sec}.

\begin{proof}(of Theorem~\ref{sewing})
Run Algorithm \emph{Recognize $S$-Orders}  on poset $P$.  Since all forcing cycles in $P$ have value at most 0, the algorithm cannot terminate during the \emph{Labeling Loop} by Proposition~\ref{trace-thm}.  By hypothesis, the algorithm cannot terminate during Procedure \emph{Assign Types}.  Thus, by Theorem~\ref{alg-correct}, the algorithm produces an $S$-representation of $P$.  \qed
 \end{proof}

\section{Appendix}

\label{appendix}

In Example~\ref{F-example} we showed that the posets $\twotwo$, $\threeone$, $\fourone$, $V$, and $Z$ are positioned correctly in the Venn diagrams of Figures~\ref{Venn-2-fig} and \ref{Venn-3-fig}.  In this appendix, we provide proofs for the remaining posets that appear in those figures.  
  
\noindent  {\bf The poset $\threeoneone$. }  This poset has the following forcing cycle ${\cal C}:  x \prec y \prec z \parallel u \parallel x \prec y \prec z \parallel v \parallel x$ with $val({\cal C}) = 0$.    Suppose $\threeoneone$ were an $S$-order and without loss of generality fix an $S$-representation  $\cal I$ of it in which all intervals have length 2 and $c(x) = 0$.  Now Theorem~\ref{wd-thm} implies that $c(y) = 1, c(z) = 2, c(u) = 1,c(v) =1$.  There are three unit intervals with center at 1,  thus if $|S| \le 2$,  two of these intervals must be identical.  Identical intervals result in twins, so the poset $\threeoneone$ can not be induced in any twin-free $S$-order when $|S| \le 2$.
 
 We also consider cases where $|S| = 3$.  A representation is possible if $S= \{A,C,D\}$ (namely by making $I_x,I_y,I_z$ of type $C$, $I_u$ of type $A$ and $I_v$ of type $D$) and  if $S= \{A,B,C\}$ (namely by making $I_x,I_u,I_z$ of type $A$, $I_y$ of type $B$ and $I_v$ of type $C$).  However,   an $S$-representation is not possible  for $S= \{B,C,D\}$, as we now show.  Since $u \parallel z$ and $v \parallel z$, by Observation~\ref{obs-centers}, we know neither $I_u$ nor $I_v$ can be of type $B$, so without loss of generality, $I_u$ is Type $C$ and $I_v$ is type $D$.  Now if $I_z$ is type $B$ or $C$ we get $v \prec z$, a contradiction, and if $I_z$ is type $D$ we get $u \prec z$, a contradiction.

\medskip
 
\noindent {\bf   The poset $H$. }  This poset has the  forcing cycle ${\cal C}:  x \prec y \prec z  \prec w \parallel u \parallel v  \parallel x$ with $val({\cal C}) = 0$.  Suppose $H$ were an $S$-order for some $S$.   Without loss of generality, fix an $S$-representation  $\cal I$ of it which all intervals have length 2 and     $c(x) = 0$.   Now Theorem~\ref{wd-thm} implies that $c(y) = 1, c(z) = 2, c(w) = 3, c(u) = 2, $ and 
$c(v) =1$.   By choosing $I_y$ and $I_z$ to be type $B$ and the remaining intervals as type $A$, we get an $AB$-representation of $H$.  

By Observation~\ref{obs-centers}, intervals $I_x, I_w, I_u, I_v$ cannot be of type $B$.  If no intervals are type $B$, then without loss of generality, $I_x, I_y, I_z, I_w$  are all of type $C$ and because $y \prec u$ and $v \prec z$, this forces $I_u, I_v$ also to be of type $C$, a contradiction.  If no interval is of type $A$, without loss of generality, $I_x$ is of type $C$, forcing $ I_y, I_z, I_w$ to also be of type $C$, and then it is not possible to assign a type to $I_u$ because $y \prec u$ and $u \parallel w$.  Thus $H$ has an $S$-representation if and only if $A,B \subseteq S$.
 
\medskip

\noindent {\bf   The poset $D$.    }
  This poset has the following induced forcing cycle ${\cal C}:  x \prec y \prec z \parallel v \parallel x \prec w \prec z \parallel v \parallel x$ with $val({\cal C}) = 0$.  Suppose $D$ were an $S$-order for some $S$.   Without loss of generality, fix an $S$-representation  $\cal I$ of it which all intervals have length 2 and     $c(x) = 0$. Now Theorem~\ref{wd-thm} implies that $c(y) = 1, c(z) = 2, c(w) = 1, c(v) =1$. 
  
  A representation is possible if $S= \{A,B,C\}$  or $S= \{B,C, D\}$ (namely by making $I_v,$ of type $A$ or type $D$,   $I_y$ of type $B$ and $I_x,I_z,I_w$ of type $C$).   Next we show a representation is not possible for $S= \{A,C,D\}$.  If it were, then without loss of generality, intervals $I_x$, $I_y$, and $I_z$ are all of type $C$ and $I_w$ is of type $A$ or $D$.  It is then impossible to have $x \prec w$ in $P$.

  Finally, consider $S$ with $|S| \le 2$.   There are three unit intervals with center at 1,  thus if $|S| \le 2$, two of these intervals must be identical.  Identical intervals result in twins, so the poset $D$ can not be induced in any twin-free $S$-order when $|S| \le 2$.
 
 \medskip
 
\noindent {\bf   The poset $Y$. }  We show that poset $Y$ is an $S$-order for $S = \{A,B,C\}$ but not when $S = \{A,C,D\}$  or $\{B,C, D\}$, nor for any $S$ with $|S| \le 2$.  The same is true of its dual.
 
   Poset $Y$ has the following induced forcing cycle ${\cal C}:  x \prec y \prec z \parallel v \parallel x \prec y \prec w \parallel v \parallel x$ with $val({\cal C}) = 0$.  Suppose $Y$ were an $S$-order for some $S$.   Without loss of generality, fix an $S$-representation  $\cal I$ of it which all intervals have length 2 and     $c(x) = 0$.   Now Theorem~\ref{wd-thm} implies that $c(y) = 1, c(z) = 2, c(w) = 2, c(v) =1$. 
   
   A representation is possible if $S= \{A,B,C\}$  (namely by making $I_x, I_z, I_v,$ of type $A$, $I_y$ of type $B$ and $I_w$ of type $C$). 
  If $\cal I$ contains only intervals of type $A$ and $B$, then $ I_z$ and $I_w$ must  be of type $A$ by Observation~\ref{obs-centers}, resulting in $z$ and $w$ being twins.  
  
  Next we consider $S = \{A,C,D\}$ and for a contradiction, assume an $S$-representation is possible.  By Observation~\ref{obs-centers}, intervals $I_y,I_z, I_w$ must all be of type $C$ or all of  type $D$ and each of these cases leads to $I_z$, $I_w$ getting identical intervals.  Similarly, if $S = \{B,C,D\}$,  intervals $I_v, I_w, I_z$ must be type $C$ or $D$ by Observation~\ref{obs-centers}, but $z$ and $w$ are both incomparable to $v$, so they are forced to get identical intervals, a contradiction.
  
  Finally, $Y$ is not induced in a twin-free $S$ order for $|S|\le 2$ since we have shown this to be true for $S=\{A,B\}$ and each other such $S$ is a subset of $\{B,C,D\}$ or $\{A,C,D\}$.

   The same is true of its dual.
   
    \medskip
   
\noindent  {\bf  The poset $X_1$.}  This poset has the forcing cycle ${\cal C}:  x \prec t \prec w  \parallel y \parallel u \prec t \prec z  \parallel v \parallel x$ with $val({\cal C}) = 0$     Suppose $X_1$ were an $S$-order  for some $S$.  Without loss of generality, fix an $S$-representation  $\cal I$ of it in which all intervals have length 2 and   $c(x) = 0$.  Now Theorem~\ref{wd-thm} implies that $c(y) = 1, c(z) = 2, c(u) = 0, c(v) = 1, c(w) = 2, c(t) = 1$.  
 
 An $S$-representation is possible for $S= \{B,C,D\}$ (namely by making $I_x, I_y, I_z$ of type $C$, $I_u,I_v,I_w$ of type $D$ and $I_t$ of type $B$).  We show that $X_1$ is not an $S$-order for any other $S$, $|S| \leq 3$. 
 
Note that $x, y, u, v$ induces a $\twotwo$ in $X_1$.  As seen in Example 10, without loss of generality $I_x, I_y$ are type $C$ and $I_u, I_v$ are  type $D$. Since Since $u \prec t$ and $v \prec t$, interval $I_t$ must be type $B$.

\medskip

\noindent  {\bf  The poset $X_2$.}  This poset has the forcing cycle ${\cal C}:  x \prec y \prec z  \parallel t \parallel u \prec v \prec w  \parallel t \parallel x$ with $val({\cal C}) = 0$     Suppose $X_2$ were an $S$-order  for some $S$.  Without loss of generality, fix an $S$-representation  $\cal I$ of it in which all intervals have length 2 and     $c(x) = 0$.  Now Theorem~\ref{wd-thm} implies that $c(y) = 1, c(z) = 2, c(u) = 0, c(v) = 1, c(w) = 2, c(t) = 1$.  
 
 An $S$-representation is possible for $S= \{A,C,D\}$ (namely by making  $I_x, I_y, I_z$  type $C$,   $I_u,I_v,I_w$  type $D$, and $I_t$  type $A$).  We show that $X_2$ is not an $S$-order for any other $S$, $|S| \leq 3$. 
 
As in the case of poset $X_1$, $x, y, u, v$ induces a $\twotwo$ in $X_2$. Again, without loss of generality $I_x, I_y$ are type $C$ and $I_u, I_v$ are  type $D$.  Since $u \parallel t$ and $x \parallel t$, interval $I_t$ must be type $A$.

\medskip

\noindent  {\bf  The poset $X_3$.} The poset $X_3$ contains both posets $X_1$ and $X_2$ and hence is not an $S$-order for $|S| \leq 3$.  A representation is possible for $S= \{A,B,C,D\}$ by starting with the $BCD$-representation for $X_1$ given above and introducing an addition interval, of type $A$, centered at 1.

\end{document}